\documentclass[12pt]{amsproc}

\usepackage{setspace}
\usepackage{latexsym}
\usepackage{amsmath, amssymb, amsfonts, amsthm}
\usepackage{float}
\usepackage{color}
\usepackage{mathrsfs}
\usepackage{hyperref}
\usepackage{enumerate}
\hypersetup{citecolor=red, linkcolor=blue, colorlinks=true}

\usepackage{bm}
\usepackage[margin=1in]{geometry}
\usepackage{url}

\newtheorem{thm}{Theorem}[section]
\newtheorem{lem}[thm]{Lemma}

\newtheorem{defi}[thm]{Definition}

\newcommand\BH{\mathrm{BH}}

\newcommand\Q{{\mathbb{Q}}}

\newcommand{\calM}{\mathcal{M}}

\newcommand{\calH}{\mathcal{H}}

\title{Homomorphisms of matrix algebras and constructions of Butson-Hadamard matrices}

\author{Padraig \'O Cath\'ain}
\address{Department of Mathematical Sciences, Worcester Polytechnic Institute, 100 Institute Road
Worcester, MA, 01609-2280, USA}

\email{pocathain@wpi.edu}

\author{Eric Swartz}
\address{Department of Mathematics, College of William \& Mary, P.O. Box 8795, Williamsburg, VA 23187-8795, USA}
\email{easwartz@wm.edu}

\begin{document}

\begin{abstract}
An $n \times n$ matrix $H$ is Butson-Hadamard if its entries are $k^{\text{th}}$ roots of unity and it satisfies $HH^* = nI_n$. Write $\BH(n, k)$ for the set of such matrices.

Suppose that $k = p^{\alpha}q^{\beta}$ where $p$ and $q$ are primes and $\alpha \geq 1$. A recent result of {\"O}sterg{\aa}rd and Paavola uses a matrix $H \in \BH(n,pk)$ to construct $H' \in \BH(pn, k)$. We simplify the proof of this result and remove the restriction on the number of prime divisors of $k$. More precisely, we prove that if $k = mt$, and each prime divisor of $k$ divides $t$, then we can construct a matrix $H' \in \BH(mn, t)$ from any $H \in \BH(n,k)$.
\end{abstract}

\maketitle

\section{Introduction}

A \textit{Butson-Hadamard matrix} $H \in \BH(n,k)$ is a matrix whose entries are all complex $k^\text{th}$ roots of unity satisfying $HH^* = nI_n$, where $H^*$ denotes the conjugate transpose of $H$ and $I_n$ is the $n \times n$ identity matrix.  It follows from the definition that distinct rows of a Butson-Hadamard matrix are orthogonal.  Hadamard proved that a matrix $M$ with real entries of modulus bounded by $1$ has determinant at most $n^{n/2}$ and that this bound is met if and only if $M$ is a \textit{real Hadamard matrix}, i.e., if and only if $M \in \BH(n,2)$ \cite{Hadamard1893}.   Butson-Hadamard matrices are so-named for their appearance in \cite{Butson}, where Butson constructed matrices in $\BH(2p, p)$.

A major open question in the theory of Butson-Hadamard matrices is for which pairs of integers $(n, k)$ the set $\BH(n,k)$ is nonempty. Recently a number of authors have used matrices in $\BH(n,k)$ to construct matrices in $\BH(mn,t)$ for various values of $(m,n,t)$. The first result of this type is due to Cohn, who proved that the existence of a matrix in $\BH(n,4)$ implies that $\BH(2n, 2)$ is nonempty \cite{Cohn1965}. More recently, Compton, Craigen and de Launey proved that a matrix in $\BH(n, 6)$ with no entries in $\{1, -1\}$ implies that $\BH(4n,2)$ is nonempty \cite{CCdeL}.  The first author and Egan unified and generalised these results, giving sufficient conditions for the existence of a matrix in $\BH(n,k)$ and a matrix in $M \in \BH(m, \ell)$ to imply the existence of a matrix in $\BH(mn, \ell)$ \cite{mypaper-morphisms}. The most substantial conditions in these constructions are on the spectrum of the matrix $M$. In \cite{mypaper-spectra}, the authors of this paper and Egan proved the existence of a real Hadamard matrix with minimal polynomial $\Phi_{2^{t+1}}(x)$, which implies that, whenever there exists $H \in \BH(n, 2^t)$, there exists a real Hadamard matrix of order $2^{2^{t-1} - 1}n$.

Perhaps the strongest of these recent results is due to {\"O}sterg{\aa}rd and Paavola \cite{OstergardButson}, who prove that a matrix in $\BH(n, pk)$ can be used to construct a matrix in $\BH(pn, k)$, provided that $k = p^{\alpha}q^{\beta}$ for primes $p$ and $q$ with $\alpha \geq 1$. In this note, we will generalise this result, removing the restriction on the number of primes dividing $k$, while also simplifying the proof.

\begin{thm}
 \label{thm:intermediate}
 Suppose that $H \in \BH(n,k)$, and $p$ is a prime such that $p^2 \mid k$. Then there exists $H' \in \BH(np, k/p)$.
\end{thm}

This result generalizes easily to the following.

\begin{thm}
\label{thm:main}
Let $k = mt$, where each prime divisor of $k$ also divides $t$.
If there exists a matrix in $\BH(n,k)$, then there exists a matrix in $\BH(mn, t)$.
\end{thm}

This paper is organized as follows.  Section \ref{sect:main} is dedicated to the proof of Theorem \ref{thm:intermediate}, from which Theorem \ref{thm:main} essentially follows immediately.  We conclude with a short discussion of obstructions to generalising Theorem \ref{thm:main} further by similar techniques.

\section{Proof of the main result}
\label{sect:main}

Throughout this section, we will use the following conventions: $k$ is a fixed positive integer and $p$ is a prime such that $p^2$ divides $k$. For any positive integer $t$, we write $\zeta_{t}$ for a primitive $t^{\textrm{th}}$ root of unity.

Given a field $K$, we say that $a \in K$ is a $p^{\textrm{th}}$ power (in $K$) if there exists $x \in K$ such that $x^{p} = a$, and we write $K^{p}$ for the set of $p^{\textrm{th}}$ powers in $K$.
We will require the following sufficient condition for the polynomial $x^n - a$ to be irreducible over $K$.

\begin{lem}\cite[Chapter VI, Theorem 9.1]{Lang}
\label{lem:Lang}
 Let $K$ be a field, and let $n \geq 2$ an integer with prime divisors $p_{1}, p_{2}, \ldots, p_{d}$. Let $a \in K / \cup_{i=1}^{d} K^{p_{i}}$. If $4 \mid n$ we also require that $a \notin -4K^4$.
 Then $x^n - a$ is irreducible in $K[x]$.
\end{lem}

Let $t = k/p$. The following $p \times p$ matrix features prominently in our main theorem, so we record its definition below.
\[M_{k,p} = \left[\begin{matrix}
 0 & 0 & \cdots & 0 & \zeta_t\\
 1 & 0 & \cdots & 0 & 0\\
 0 & 1 & \cdots & 0 & 0\\
   &   & \ddots &   &  \\
 0 & 0 & \cdots & 1 & 0
\end{matrix}\right].\]

\begin{lem}
\label{lem:psiisom}
Let $p$ be a prime divisor of the positive integer $k$, and let $\zeta_k$ be a primitive $k^{\text{th}}$ root of unity.
Write $t = k/p$ and set $\zeta_{t}$ to be a primitive $t^{\textrm{th}}$ root of unity. Then
\[\psi_{k,p}: \zeta_k^i \mapsto M_{k,p}^i,\]
defines an isomorphism of $\mathbb{Q}$-algebras, $\psi_{k,p}: \mathbb{Q}[\zeta_{k}] \rightarrow \mathbb{Q}[M_{k,p}]$.
\end{lem}

\begin{proof}
 If $t = k/p$, the field $\Q[\zeta_k]$ is a $p$-dimensional extension of $\Q[\zeta_t]$ and is isomorphic to $\Q[x]/ \langle f(x) \rangle$, where $f(x) = x^p - \zeta_t$.  Note that $f(x)$ is irreducible over $\Q[\zeta_t]$ by Lemma \ref{lem:Lang}.  The minimal polynomial of the matrix $M_{k,p}$ is also $f(x)$, and so $\Q[\zeta_k] \cong \Q[M_{k,p}]$.
\end{proof}

Next, we will lift this embedding of fields to an embedding of matrix algebras.

\begin{lem}
\label{lem:psihom}
 Let $\psi$ be a function defined entrywise on a matrix in $\calM_n(\Q[\zeta_k])$ by
   \[A^{\psi} := \left(a_{i,j}^{\psi_{k,p}} \right),\]
   so that $A^{\psi} \in \calM_{np}(\Q[\zeta_t])$. Then $\psi$ is a homomorphism of matrix algebras.
\end{lem}

\begin{proof}
 Let $A = (a_{i,j}), B = (b_{i,j}) \in \calM_n(\Q[\zeta_k])$.  It suffices to check that $\psi$ preserves matrix addition, the conjugate transpose operation $*$, and matrix multiplication.  Then, using Lemma \ref{lem:psiisom},
 \begin{align*}
  (A + B)^\psi &= \left((a_{i,j} + b_{i,j})^{\psi_{k,p}} \right) = \left(a_{i,j}^{\psi_{k,p}} + b_{i,j}^{\psi_{k,p}} \right) = \left(a_{i,j}^{\psi_{k,p}}\right) + \left(b_{i,j}^{\psi_{k,p}} \right) = A^\psi + B^\psi,\\
  \left(A^*\right)^\psi &= \left((a_{i,j})^*\right)^\psi = \left((\overline{a_{j,i}})\right)^\psi = \left(\overline{a_{j,i}}^{\psi_{k,p}} \right) =  \left(a_{i,j}^{\psi_{k,p}} \right)^* = \left(A^\psi\right)^*,\\
  \left(AB\right)^\psi &= \left( \sum_{\ell=1}^{n} a_{i,\ell}b_{\ell,j} \right)^\psi = \left(\left(\sum_{\ell=1}^{n} a_{i,\ell}b_{\ell,j}\right)^{\psi_{k,p}} \right)= \left(\sum_{\ell = 1}^n a_{i,\ell}^{\psi_{k,p}} b_{\ell,j}^{\psi_{k,p}} \right) = A^\psi B^\psi,
 \end{align*}
as desired.
\end{proof}
It follows from Lemma \ref{lem:psihom} that if $M, N \in \calM_n(\Q[\zeta_k])$ are such that $MN^{\ast} = I_{n}$ then $M^{\psi}(N^{\ast})^{\psi} = (MN^{\ast})^{\psi} = I_{nk}$.

It is well-known that $\BH(p,p)$ is always nonempty. It contains, for example, the Discrete Fourier Matrix of order $p$:
\[ F_{p} = \left( \zeta_p^{ij} \right)_{0\leq i,j \leq p-1}\]
where $\zeta_p$ denotes a primitive $p^{\text{th}}$ root of unity. In general, character tables of abelian groups give rise to Butson-Hadamard matrices, and the Discrete Fourier Matrix is such an example for a cyclic group of order $p$.

\begin{defi}
\label{def:psi}
If $\psi$ is the entrywise application of $\psi_{k,p}$, $H \in \BH(n,k)$, and $C \in \BH(p,p)$, then we define
\[ \calH(H,C, \psi) := H^\psi \cdot (I_n \otimes C),\]
where $I_n \otimes C$ is the Kronecker product.
\end{defi}

Note that by definition the matrix $I_n \otimes C$ will be a block diagonal matrix with each block equal to $C$.  We will now prove Theorem \ref{thm:intermediate}, which generalises \cite[Theorem 4]{OstergardButson}.

\begin{proof}[Proof of Theorem \ref{thm:intermediate}]
Let $k$ be a positive integer such that $p^2 \mid k$ for some prime $p$ and let $t = k/p$.  Let $\psi:= \psi_{k,p}$, $H \in \BH(n,k)$, and $C \in \BH(p,p)$.  We will show that $\calH(H,C,\psi) \in \BH(np, t)$. Observe first that $\calH(H, C, \psi)$ is a block matrix, and by definition every block is of the form $M_{k,p}^i C$ for some $i$.  Since $M_{k,p}$ is a monomial matrix whose nonzero entries are $t^\textrm{th}$ roots of unity and all the entries of $C$ are $p^\textrm{th}$ roots of unity, the entries of $\calH(H,C, \psi)$ are all products of $t^\textrm{th}$ and $p^\textrm{th}$ roots of unity.  Since by assumption $p^2 \mid k$, we have $\zeta_p \in \Q[\zeta_t]$, and hence every entry of $\calH(H,C, \psi)$ is a $t^\textrm{th}$ root of unity. It now suffices to check that the Gram matrix has the required form. Using Lemma \ref{lem:psihom}, we have
  \begin{align*}
      \calH(H,C,\psi) \cdot \calH(H,C, \psi)^* &= H^\psi \cdot (I_n \otimes C) \cdot \left( H^\psi \cdot (I_n \otimes C) \right)^* \\
  &= H^\psi \cdot (I_n \otimes C) \cdot (I_n \otimes C)^* \cdot \left( H^\psi \right)^* \\
  &= H^\psi \cdot (I_n \otimes C) \cdot (I_n \otimes C^*) \cdot \left( H^\psi \right)^* \\
  &= H^\psi \cdot (I_n \otimes CC^*) \cdot \left( H^\psi \right)^*\\
  &= H^\psi \cdot (I_n \otimes p\cdot I_p) \cdot \left(H^\psi \right)^*\\
  &= p \cdot H^\psi \cdot \left(H^\psi \right)^* \\
  &= p \cdot \left(H \cdot H^* \right)^\psi \\
  &= p \cdot \left( n\cdot I_n\right)^\psi \\
  &= np \cdot I_{np}.
  \end{align*}
The result follows.
\end{proof}

We can now prove Theorem \ref{thm:main}.

\begin{proof}[Proof of Theorem \ref{thm:main}]
Let $k = mt$, and write $m = p_{1}p_{2} \ldots p_{d}$ as a product of (not necessarily distinct) primes. By hypothesis, $\BH(n,k)$ is nonempty. Applying Theorem \ref{thm:intermediate} repeatedly, we obtain matrices in the sets $\BH(n \prod_{i=1}^{c} p_{i}, t \prod_{i=c+1}^{d} p_{i} )$ for $c = 1, 2, \ldots, d$.
\end{proof}

Theorems \ref{thm:intermediate} and \ref{thm:main} do not allow us to eliminate a prime divisor from $k$. If one constructs a matrix $M_{k,p}$, where $p$ divides $k$ and $p^2$ does not, then, since $p$ is coprime to $k/p$, the map $\zeta \mapsto \zeta^p$ permutes the primitive $(k/p)^{\textrm{th}}$ roots of unity, so there exists $\zeta'$ such that $(\zeta')^{p} - \zeta_{k/p} = 0$, and the minimal polynomial of $M_{k,p}$ has a linear factor. Then we do not obtain a homomorphism in Lemma \ref{lem:psiisom}, and none of the remaining proof follows. It is instructive to consider the problem of constructing homomorphisms $\psi: \mathbb{Q}[\zeta_{p}] \rightarrow \mathbb{Q}$ (equivalent to constructing morphisms from $\BH(n, p)$ onto real Hadamard matrices); as of writing, the only examples known have $k = 2^{\alpha}$ for some $\alpha$. This is because the only irreducible polynomials over $\mathbb{Q}$ with precisely two non-zero terms (and constant term $\pm 1$) are of the form $x^{2^{\alpha}} +1$.

Parity obstructions also arise: suppose that we had a monomial matrix $P$ with minimal polynomial $\frac{x^{p}-1}{x-1}$, then we could have
\[ \sum_{i = 0}^{p-1} P^i \cdot C = \left( \sum_{i = 0}^{p-1} P^i \right) \cdot C = 0.\]
But each $P^{i}C$ is a Hadamard matrix, and the number of terms in the sum is odd, so no entry in the matrix can vanish. One could perhaps circumvent the restriction of binomial minimal polynomials by considering arbitrary matrices $K_{1}, \ldots, K_{t}$ and $L_{1}, \ldots, L_{t}$ such that $K_{i}L_{j}^{\ast}$ is complex Hadamard for each pair $1 \leq i,j\leq t$; such constructions have been considered by McNulty and Weigert \cite{McNulty} for example. We leave this as a direction for future research.\\

\noindent\textsc{Acknowledgements.}  The authors would like to thank the anonymous referees for many useful suggestions that helped improve this paper.

\bibliographystyle{plain}
\bibliography{Biblio2018}

\end{document}